\documentclass{amsart}

\usepackage[inactive]{srcltx} 
\usepackage{amsmath, amsthm, amscd, amsfonts, amssymb, graphicx, color, extarrows}
\usepackage[numbers, sort & compress]{natbib}
 \newtheorem{thm}{Theorem}[section]
 \newtheorem{cor}[thm]{Corollary}
 \newtheorem{lem}[thm]{Lemma}
 \newtheorem{prop}[thm]{Proposition}
 \theoremstyle{definition}
 
 \theoremstyle{remark}
 \newtheorem{rem}[thm]{Remark}
 \numberwithin{equation}{section}

\newcommand{\A}{A^{**}}

\begin{document}

\title[Linear maps behaving like derivations or anti-derivations...]
 {Linear maps behaving like derivations or anti-derivations at orthogonal elements on $ C^\star $-algebras}

\author{ B. Fadaee \quad  and \quad H. Ghahramani}
\thanks{{\scriptsize
\hskip -0.4 true cm \emph{MSC(2010)}: 46L05; 47B47; 46L57.
}}

\address{Department of
Mathematics, University of Kurdistan, P. O. Box 416, Sanandaj,
Iran.}
\email{behroozfadaee@yahoo.com; b.fadaee@sci.uok.ac.ir}

\address{Department of
Mathematics, University of Kurdistan, P. O. Box 416, Sanandaj,
Iran.}
\email{h.ghahramani@uok.ac.ir; hoger.ghahramani@yahoo.com}


\address{}

\email{}

\thanks{}

\thanks{}



\date{}

\dedicatory{}

\commby{}

\begin{abstract}
Let $A$ be a $ C^\star $-algebra and $\delta:A \rightarrow \A$ be a continuous linear map. We assume that $\delta$ acts like derivation or anti-derivation at orthogonal elements for several types of orthogonality conditions such as $ab=0$, $ab^{\star}=0$, $a^{\star}b=0$, $ab=ba=0$ and $ab^{\star}=b^{\star}a=0$. In each case, we characterize the structure of $\delta$. Then we apply our results for von Neumann algebras and unital simple $C^{\star}$-algebras.
\end{abstract}

\maketitle
\section{ Introduction}
Algebras and vector spaces in this paper are assumed to be those over the complex field $\mathbb{C}$. Let $A$ and $M$ be an algebra and an $A$-bimodule, respectively. Recall that a linear map $d:A\rightarrow M$ is said to be a \textit{derivation} if $d(ab) = ad(b)+d(a)b $ for all $a, b \in A$. Also, $d$ is called inner derivation if for some $m \in M$, $d$ takes the form $d(a)=am-ma$ for all $a \in A$. We also call $d$ \textit{anti-derivation} if $d(ab) =bd(a)+d(b)a$ for all $a,b \in A$. Derivation is an important field of research, and have been studied and applied extensively both in theory and applications. For this and related topics, see \cite{da, gh1, gh12} among others. 
\par 
There are a number of papers investigating the conditions under which mappings of (Banach) algebras are thoroughly determined by actions on some sets of points. We refer the reader to \cite{al1, bre, chb, gh1, gh2, h5, h6, gh5} for a full account of the topic and a list of references. The following condition has been drawing many researchers attention working in this field:
\[a, b \in A, \,\, ab=z\Rightarrow \delta(ab) = a\delta(b)+\delta(a)b \quad (\blacklozenge),\]
where $z\in A$ is fixed and $\delta:A\rightarrow M$ is a linear (additive) map. 
Bre$\check{\textrm{s}}$ar, \cite{bre} studied the derivations of rings with idempotents with $z=0$. In \cite{bre} it was demonstrated that if $A$ is a prime ring containing a non-trivial idempotent and $\delta:A\rightarrow A$ is an additive map satisfying $(\blacklozenge)$ with $z=0$, then $\delta(a)=d(a)+ca$ ($a\in A$) where $d$ is an additive derivation and $c$ is a central element of $A$. Note that the nest algebras are important operator algebras that are not prime. Jing et al. in \cite{jin} showed that, for the cases of nest algebras on a Hilbert space and standard operator algebras in a Banach space, the set of linear maps satisfying $(\blacklozenge)$ with $z=0$ and $\delta(I)=0$ coincides with the set of inner derivations. Then, many studies have been done in this case and different results were obtained, for instance, see \cite{al1, al10, bre, chb, ess, gh1, gh12, zhang1, zh1, zh3, zh4} and the references therein. Recently, in \cite{che} the additive maps on a prime ring acting on some orthogonality condition are described where the ring has an involution and non-trivial idempotents. Also in \cite{al10} the authors considered the subsequent condition on a continuous linear map $\delta$ from a $C^\star$-algebra $A$ into an essential Banach $A$-bimodule $M$: 
\[ a,b\in A, \quad ab=ba=0 \Longrightarrow a\delta(b)+\delta(a)b+b\delta(a)+\delta(b)a=0,\] 
and they showed that there exist a derivation $d:A\rightarrow M$ and a bimodule homomorphism $\Phi:A\rightarrow M$ such that $\delta=d+\Phi$. Motivated by these reasons, in this paper we consider the problem of characterizing continuous linear maps on $C^\star$-algebras behaving like derivations or anti-derivations at orthogonal elements for several types of orthogonality conditions. 
\par 
In this paper we consider the problem of characterizing continuous linear maps behaving like derivations or anti-derivations at orthogonal elements for several types of orthogonality conditions on $C^\star$-algebras. In particular, in this paper we consider the subsequent conditions on a continuous linear map $\delta:A\rightarrow \A$ where $A$ is a  $C^\star$-algebra:
\begin{enumerate}
\item[(i)] \textit{derivations through one-sided orthogonality conditions}
 \[ ab=0 \Longrightarrow a\delta(b)+\delta(a)b=0 ;\]
 \[ab^{\star}=0 \Longrightarrow a\delta(b)^{\star}+\delta(a)b^{\star}=0 ;\]
 \[a^{\star}b=0 \Longrightarrow a^{\star}\delta(b)+\delta(a)^{\star}b=0 ;\]
\item[(ii)] \textit{anti-derivations through one-sided orthogonality conditions}
 \[ ab=0 \Longrightarrow b\delta(a)+\delta(b)a=0 ;\]
 \[ ab^{\star}=0 \Longrightarrow \delta(b)^{\star}a+b^{\star}\delta(a)=0;\]
 \[a^{\star}b=0 \Longrightarrow \delta(b)a^{\star}+b\delta(a)^{\star}=0;\]
\item[(iii)] \textit{Derivations through two-sided orthogonality conditions}
 \[ab=ba=0 \Longrightarrow a\delta(b)+\delta(a)b=b\delta(a)+\delta(b)a=0;\]
\[ ab^{\star}=b^{\star}a=0 \Longrightarrow a\delta(b)^{\star}+\delta(a)b^{\star}=\delta(b)^{\star}a+b^{\star}\delta(a)=0;\]
\end{enumerate}
where $a,b \in A$. Our purpose is to investigate whether the above conditions characterize 
continuous derivations ($\star$-derivations) or continuous anti-derivations ($\star$-anti-derivations) on $C^\star$-algebras. Also we give applications of our results for von Neumann algebras and unital simple $C^{\star}$-algebras.
\par 
The following are the notations and terminologies which are used throughout this article.
\par 
Let $A$ be a $C^\star$-algebra. For $\mu , \nu \in \A$, we denote by $\mu\nu$ the first Arens product. In view of the fact that any $C^\star$-algebra is Arens regular, the product $\mu\nu$ in $\A$ coincide with the second Arens product. By \cite[Theorem 2.6.17]{da} the product $\mu\nu$ in $\A$ is separately continuous with respect to the weak$^*$ topology ($\sigma(\A,A^{*})$) on $\A$. The Banach algebra $\A$ is unital by this product and the unity of $\A$ is denoted by $e$. 
\par 
If a net $(\lambda_{i})_{i\in I}$ in $\A$ converges to $\lambda \in \A$ with respect to the weak$^{*}$ topology, we write it by $\lambda_{i}\xlongrightarrow{\sigma(\A,A^{*})} \lambda$.
\begin{rem}\label{11}
Let $A$ be a $ C^\star $-algebra, and let $\mu \in \A$. If $\mu   A =\lbrace0\rbrace$ or $A \mu=\lbrace0\rbrace$, then $\mu =0$. 
\begin{proof}
By Goldstine Theorem there is a net $(a_{i})_{i\in I}$ in $A$ such that $a_{i}\xlongrightarrow{\sigma(\A,A^{*})} e$, where $e$ is the unity of $\A$. Let $\mu A= \lbrace0\rbrace$. By separately $w^{*}$-continuity of product in $\A$ we have $\mu a_{i}\xlongrightarrow{\sigma(\A,A^{*})} \mu e=\mu$. Hence $\mu=0$. Similarly, we can show that $A \mu=\lbrace0\rbrace$ implies $\mu=0$.
\end{proof}
\end{rem}
We note that the centre of an algebra $A$ is written by $\mathcal{Z}(A)$.
\begin{rem}\label{12}
Let $A$ be a $ C^\star $-algebra, and let $\mu \in \A$.
Suppose that $\mu   a= a   \mu$ for each $a \in A$. Then $\mu \in \mathcal{Z}(\A)$.
\begin{proof}
 Let $ \nu \in \A $. We will show that $ \mu \nu = \nu \mu $. By Goldstine Theorem there is a net $(a_{i})_{i\in I}$ in $ A $ such that $ a_{i} \xlongrightarrow{\sigma (\A, A^*)} \nu $. By separately $w^{*}$-continuity of product in $\A$ we have
\[  \mu a_{i} \xlongrightarrow{\sigma (\A, A^*)} \mu\nu \quad \text{and} \quad a_{i}\mu \xlongrightarrow{\sigma (\A, A^*)} \nu \mu  .\]
By applying assumption, $ \mu \nu = \nu \mu $.
\end{proof}
\end{rem}
\begin{rem}\label{13}
Let $A$ be $ C^\star $-algebra and $(u_{i})_{i\in I}$ be a bounded approximate identity of $A$. Since $(u_{i})_{i\in I}$ is bounded, by Banach-Alaoglu Theorem we can assume that it converges to some $\mu \in \A$ with respect to the weak$^{*}$ topology. So by separately $w^{*}$-continuity of product in $\A$ we have $u_{i} a\xlongrightarrow{\sigma (\A, A^*)} \mu a$ for all $a\in A$. On the other hand by the fact that $(u_{i})_{i\in I}$ is an approximate identity, for each $a\in A$ we get $u_{i} a\xlongrightarrow{\sigma (\A, A^*)} a$ in $\A$. So $(\mu-e)A=\lbrace 0 \rbrace $ and by Remark \ref{11}, it follows that $\mu=e$. Therefore we can assume that the $ C^\star $-algebra $A$ has a bounded approximate identity such that $u_{i}\xlongrightarrow{\sigma (\A, A^*)} e$ in $\A$.
\end{rem}
Let $h :A\rightarrow\A$ be a map. We say that $h$ is a $\star$-map whenever $h(a^{\star})=h(a)^{\star}$ for all $a\in A$.
\section{Derivations and anti-derivations at orthogonality product}
In this section we will consider a continuous linear map on a $ C^{\star}  $-algebra behaving like derivation or anti-derivation at one-sided orthogonality conditions. \par 
In order to prove our results we need the following result for continuous bilinear maps on  $ C^{\star}  $-algebras.
\begin{lem} (Alaminos et al. \cite{al1}).\label{l21}
Let $A$ be a $ C^{\star}  $-algebra, let $ X $ be a Banach space, and let $ \phi: A \times A \rightarrow X $ be a continuous bilinear map such that $ \phi (a, b) =0 $ whenever $ a,b \in A $ are such that $ ab=0 $, Then $ \phi ( ab ,c  ) = \phi (a, bc) $ for all
$ a,b , c \in A $. Also there is a continuous linear map $ \Phi: A \rightarrow X $ such that $ \phi (a, b ) = \Phi (ab) $  for all
$ a,b  \in A $.
\end{lem}
It should be noted that in the above lemma if $A$ is a commutative $ C^{\star}  $-algebra, then $\phi$ is symmetric, that is $\phi(a,b)=\phi(b,a)$ for all $a,b\in A$. 
\par 
Now, we characterize continuous linear maps on  $ C^{\star}  $-algebras behaving like derivations through one-sided orthogonality conditions. 
\begin{thm}\label{o1}
Let $A$ be a $ C^{\star}  $-algebra, and let $\delta: A \rightarrow \A$ be a continuous linear map.
\begin{enumerate}
\item[(i)] $ \delta $ satisfies 
\begin{equation*}
 a b=0 \Longrightarrow a   \delta(b)+\delta(a)   b =0\quad (a , b \in A).
\end{equation*}
if and only if there is a continuous derivation $ d: A \rightarrow \A $ and an element $ \eta \in \mathcal{Z}(\A) $ such that $\delta(a)= d(a) + \eta    a  $ for all $ a \in A$. 
\item[(ii)] $ \delta $ satisfies 
\begin{equation*}
 a b ^{\star}=0 \Longrightarrow a   \delta(b)^{\star}+\delta(a)   b ^{\star}=0\quad (a , b \in A).
 \end{equation*}
if and only if there is a continuous $\star$-derivation $ d: A \rightarrow \A $ and an element $ \eta \in \A $ such that $\delta(a)= d(a) + \eta    a  $ for all $ a \in A$. 
\end{enumerate}
\end{thm}
\begin{proof}
$(i)$ By \cite [Theorem 4.6]{al1} and Remark \ref{11}, there is a continuous derivation $ d : A  \rightarrow \A$ and an element $\eta \in \mathcal{Z}( \A)$ such that $\delta(a)=d(a)+\eta   a $ for all $a \in A $. The converse is proved easily.
\par 
$(ii)$ Suppose that $(u_i)_{i\in I}$ is a bounded approximate identity of $A$ such that $u_{i}\xlongrightarrow{\sigma (\A, A^*)} e $, where $e$ is the identity of $\A$. Since the net $(\delta(u_{i}))_{i\in I}$ is bounded, we can assume that it converges to some $\eta \in \A$ with respect to the weak$^{*}$ topology. Define $d: A \rightarrow \A$ by $d(a)=\delta(a)-\eta a $. Then $d$ is a continuous linear map which satisfies 
\begin{equation}\label{1o4}
 ab^{\star}=0 \Longrightarrow a d(b)^{\star}+d(a) b^{\star}=0\quad (a, b \in A),
 \end{equation}
and $d(u_{i})\xlongrightarrow{\sigma (\A, A^*)} 0$. We will show that $d$ is a  $\star$-derivation. In order to prove this we consider the continuous bilinear map $\phi: A  \times A \rightarrow \A$ by $\phi(a , b )=a   d(b^{\star})^{\star}+d(a)    b $. If $ a , b \in A$ are such that $ a b =0$, then $a (b^{\star})^{\star}=0$ and \eqref{1o4} gives $\phi( a , b )=0$. So by Lemma \ref{l21}, we get $ \phi (a b ,c)=\phi (a,bc)$ for all $a , b , c \in A$. Therefore
\begin{equation}\label{1o5}
a b  d(c^{\star})^{\star}+d(ab) c=a  d(c^{\star} b^{\star})^{\star}+d(a)   bc,
 \end{equation}
 for all $a , b , c \in A$. On account of \eqref{1o5}, for all $ b , c \in A $ we have
 \[u_{i} b   d(c^{\star})^{\star}+d(u_{i}b)   c=u_{i}   d (c^{\star} b ^{\star})^{\star}+d(u_{i})   b c.\]
From continuity of $d$, we get $u_{i}    d(c^{\star} b ^{\star})^{\star}+d(u_{i}) b c  $ converges to $ b   d(c^{\star})^{\star}+d (b)   c$ with respect to the norm topology. On the other hand, from separately $w^{*}$-continuity of product in $\A$ and $d(u_{i})\xlongrightarrow{\sigma (\A, A^*)} 0$ it follows that
$u_{i}   d (c^{\star} b^{\star})^{\star}+d(u_{i})   b c \xlongrightarrow{\sigma (\A, A^*)} d(c^{\star}   b ^{\star})^{\star}$. Hence 
\begin{equation}\label{1o6}
d( a b^{\star})=d(a)   b ^{\star}+ a   d(b)^{\star},
 \end{equation}
for all $a , b  \in A $. Now letting $ a =u_{i}$ in \eqref{1o6}, we obtain
\[ d (u_{i} b ^{\star})= d (u_{i})   b^{\star}+u_{i}   d(b)^{\star},\]
for all $ b \in A$. By continuity of $d$, $d(u_{i})\xlongrightarrow{\sigma (\A, A^*)} 0$ and using similar arguments as above it follows that $d(b^{\star})=d(b)^{\star}$ for all $b\in A$. Thus from \eqref{1o6}, $d$ is a $\star$-derivation.
\par 
The converse is proved easily.
\end{proof}

\begin{cor} \label{c32}
Let $A$ be a $ C^{\star} $-algebra, and let $\delta: A \rightarrow \A$ be a continuous linear map. Then
\begin{equation*}
 a^{\star} b=0 \Longrightarrow a^{\star}   \delta(b)+\delta(a)^{\star}   b=0\quad (a , b \in A).
 \end{equation*}
if and only if there is a continuous $\star$-derivation $ d: A \rightarrow \A $ and an element $ \eta \in \A $ such that $\delta(a)= d(a) + a \eta  $ for all $ a \in A$. 
\end{cor}
\begin{proof}
Consider the continuous linear map $ \tau : A \rightarrow \A $ defined by $ \tau (a) = \delta (a^{\star})^{\star} $. 
It is easily seen that this map satisfies the conditions of Theorem \ref{o1}-$(ii)$. 
So there exists a continuous $ \star $-derivation $ d : A \rightarrow \A $ and an element $ \eta_1 \in \A $ such that
 $ \tau (a) = d(a) + \eta_1    a  $ for all $ a \in A $. Then $ \delta(a)= d(a) + a    \eta  $ for all 
$ a \in A$, where $ \eta = \eta^{\star}_1 $.

The converse is proved easily.
\end{proof}
Note that in above corollary and part $(ii)$ of Theorem \ref{o1}, it is not necessary true that $ \eta \in \mathcal{Z} ( \A) $. For example suppose that $ \eta $ is not in $ \mathcal{Z} ( \A) $ and define $ \delta : A \rightarrow \A $ by $ \delta (a) = \eta a $. Then $ \delta $ satisfies in conditions of Theorem \ref{o1} and $ \delta $ equals to sum of the zero derivation and $ \eta a $, but $ \eta $ is not belong to $ \mathcal{Z} ( \A) $.
\par 
\begin{rem}\label{der}
Let $A$ be a $C^\star$-algebra and $d: A\rightarrow A $ be an inner derivation where $d(a)=a\mu - \mu a$ for some $\mu\in A$. If $d$ is a $\star$-map, then $a^{\star}\mu-\mu a^{\star}=\mu^{\star}a^{\star}-a^{\star}\mu^{\star}$ for all $a\in A$. So $Re\mu=\dfrac{1}{2}(\mu+\mu^{\star})\in \mathcal{Z}(A)$. Conversely for $\mu\in A$ with $Re\mu\in \mathcal{Z}(A)$, the map $d: A\rightarrow A $ defined by $d(a)=a\mu - \mu a$ is a $\star$-inner derivation. 
\end{rem}
\par 
If $ A $ is a von Neumann algebra or a simple $ C^{\star} $-algebra with unity, then by \cite[Theorem 4.1.6]{sak} and \cite[Theorem 4.1.11]{sak} every derivation $ d: A \rightarrow A $ is an inner derivation. In view of these results, we have the next proposition.
\begin{prop}
Let $ A $ be a von Neumann algebra or a simple $ C^{\star} $-algebra with unity. Suppose that $ \delta : A \rightarrow A $ is a continuous linear map. Then 
\begin{enumerate}
\item[(i)]
$ \delta $ satisfies 
\begin{equation*}
 a b=0 \Longrightarrow a \delta(b)+\delta(a)  b =0\quad (a , b \in A)
\end{equation*}
if and only if there are elements $\mu, \nu \in A$ such that $\delta(a)=a  \mu-\nu a $ for all $ a \in A$ and $\mu-\nu\in \mathcal{Z}(A)$.
\item[(ii)]
$ \delta $ satisfies
\begin{equation*}
 a b ^{\star}=0 \Longrightarrow a  \delta(b)^{\star}+\delta(a) b ^{\star}=0\quad (a , b \in A),
 \end{equation*}
if and only if there are elements $\mu, \nu\in A$ such that $\delta(a)=a  \mu-\nu b $ for all $a \in A$ and $Re\mu \in \mathcal{Z}(A)$.
\item[(iii)]
$ \delta $ satisfies
\begin{equation*}
 a^{\star} b =0 \Longrightarrow a^{\star} \delta(b)+\delta(a)^{\star}  b =0\quad (a , b \in A)
 \end{equation*}
if and only if there are elements $\mu, \nu\in A $ such that $\delta(a)=a  \nu-\mu a $ for all $a  \in A $ and $Re\mu\in \mathcal{Z}(A)$.
\end{enumerate}
\end{prop}
\begin{proof}
$ (i) $ Suppose that $ \delta $ satisfies the given condition. By Theorem \ref{o1}-$(i)$, there exists a continuous derivation $ d : A \rightarrow \A $ and an element $ \eta \in \mathcal{Z} (\A) $ such that $ \delta (a) = d(a) + \eta   a  $ for all $ a \in A$. Since $ d $ is an derivation and $ A $ is unital, we have $ d(1) = 0  $, so $ \delta (1) = \eta \in A $ and hence $ d(A) \subseteq A $, and also $ \eta \in \mathcal{Z} (A) $. Since every derivation on $ A $ is inner, it follows that $ d(a) = a \mu - \mu a $ for all $ a \in A $. Setting $ \nu = \mu - \eta $. So $ \delta (a) = a \mu - \nu a $ for all $ a \in A $ and $ \mu - \nu \in \mathcal{Z} (A) $.
\par 
The converse is proved easily.
\par 
$ (ii) $ Let $ \delta $ satisfies $ (ii) $. By Theorem \ref{o1}-$(ii)$, there exists a continuous $ \star $-derivation $ d : A \rightarrow \A $ and an element $ \eta \in \A $ such that $ \delta (a) = d(a) + \eta   a  $ for all $ a \in A$. By using similar methods as in proof of part $(i)$, $ d: A \rightarrow A $ is a $ \star $-derivation, and hence it is inner. Now by Remark \ref{der}, there exists an element $ \mu \in A $ with $ Re \mu \in \mathcal{Z} (A) $ such that $ d(a) = a \mu - \mu a $ for all $ a \in A $, where $ \mu \in A $. Setting $ \nu = \mu - \eta $. So we have $ \delta (a) = a \mu - \nu a $ for all $ a \in A $ and $ Re \mu \in \mathcal{Z} (A) $.

The converse is proved easily.
\par 
$ (iii) $ By using Corollary \ref{c32} and a similar proof as that of part $(ii)$, we can obtain this result.
\end{proof}
Let $A$ be an algebra and $M$ be an $A$-bimodule. Recall that a linear map $d:A\rightarrow M$ is said to be a \textit{Jordan derivation} if $d(a^{2})=ad(a)+d(a)a$ for all $a\in A$. Clearly, each derivation is a Jordan derivation. The converse is not true, in general. Johnson in \cite{jon} has shown that any continuous Jordan derivation from a $C^{\star}$-algebra $A$ into any Banach $A$-bimodule is a derivation.
\par 
In the next theorem we characterize anti-derivations through one-sided orthogonality conditions. 
\begin{thm}\label{anti}
Let $A$ be a $ C^{\star} $-algebra, and let $\delta:A \rightarrow \A$ be a continuous linear map.
\begin{enumerate}
\item[(i)] Assume that 
\[ a b=0 \Longrightarrow b   \delta(a)+\delta(b)   a =0\quad (a , b  \in A).\]
Then there is a continuous derivation $ d: A \rightarrow \A $ and an element $ \eta \in \mathcal{Z} ( \A ) $ such that $ \delta (a) = d (a) + \eta   a $ for all $ a \in A $.
\item[(ii)] Assume that 
\begin{equation*}
 a b ^{\star}=0 \Longrightarrow \delta(b)^{\star}   a+b^{\star}    \delta(a)=0\quad (a , b \in A).
 \end{equation*}
Then there is a continuous derivation $ d: A \rightarrow \A $ and an element $ \eta \in \A $ such that $ \delta (a) = d (a) + a \eta $ for all $ a \in A $.
\end{enumerate}
\end{thm}
\begin{proof}
Suppose that $(u_i)_{i\in I}$ is a bounded approximate identity of $A$ such that $u_{i}\xlongrightarrow{\sigma (\A, A^*)} e $, where $ e$ is the identity of $\A$.
\par 
$(i)$ Define a continuous bilinear map $\phi:A \times A \rightarrow \A$ by $\phi(a , b )=b   \delta(a)+\delta(b)   a$. Then $\phi (a , b ) = 0$ for all $a , b \in A $ with $ a b =0$. By applying Lemma \ref{l21}, we obtain $ \phi (a b ,c)=\phi (a,bc)$ for all $ a , b , c  \in A $. So 
\begin{equation}\label{anti1}
c   \delta(ab)+\delta(c)   a b =  b c \delta(a)+\delta(bc)   a ,
\end{equation}
for all $a, b , c  \in A $. Since the net $(\delta(u_{i}))_{i\in I}$ is bounded, we can assume that it converge to some $\eta \in \A$ with respect to the weak$^{*}$ topology. On account of \eqref{anti1}, for all $ a , b \in A$ we have
\[ u_{i}   \delta( a b )+\delta(u_{i})   a b=  b u_{i}   \delta(a)+\delta(bu_{i})   a.\]
From continuity of $\delta$, we get $ b u_{i}   \delta(a)+\delta(bu_{i})   a$ converges to $ b   \delta(a)+\delta(b)   a $ with respect to the norm topology. On the other hand, by separately $w^{*}$-continuity of product in $\A$, it follows that $ u_{i}   \delta(ab)+\delta(u_{i})   a b $ converges to $\delta( a b )+\eta    a b $ with respect to the weak$^{*}$ topology. Hence
\begin{equation}\label{anti2}
\delta( a b )= b   \delta(a)+\delta(b)   a -\eta   a b 
\end{equation}
for all $a , b \in A$. Now letting $a=u_{i}$ in \eqref{anti1}, we obtain
\[ b c   \delta(u_{i})+\delta(b c )   u_{i}= c   \delta(u_{i}    b)+\delta(c)    u_{i} b .\]
By this identity and using similar arguments as above it
follows that
\begin{equation}\label{anti3}
\delta( a b )= b    \delta(a)+\delta(b)    a - a b    \eta
\end{equation} 
for all $ a , b \in A$. Hence from \eqref{anti2} and \eqref{anti3}, for each $a , b  \in A$, we find that $\mu   a b = a b    \eta $. So by the fact that $A^{2}=A$ and Remark \ref{12}, it follows that $\eta \in \mathcal{Z}(\A)$. Define $d:A \rightarrow \A$ by $d(a)=\delta(a)-\eta    a $. The linear map $d$ is continuous and by \eqref{anti2} and the fact that $\eta \in \mathcal{Z}(\A)$, it follows that $d$ is an Jordan derivation. From \cite{jon}, $ d $ is a derivation. 
\par 
$(ii)$ In order to prove this we consider the continuous bilinear map $\phi:A \times A \rightarrow \A$ defined by $\phi(a , b )=\delta(b^{\star})^{\star}    a + b    \delta(a)$. If $a , b \in A$ are such that $ a b =0$, then $\phi(a , b )=0$. So by Lemma \ref{l21}, we get $ \phi (ab , c )=\phi (a , b c )$ for all $a , b , c \in A$. Therefore
\begin{equation}\label{anti4}
\delta(c^{\star})^{\star}    a  b + c    \delta(a b )=\delta(c^{\star} b^{\star})^{\star}    a + b c    \delta(a),
 \end{equation}
 for all $a , b , c  \in A$. Setting $a=u_{i}$ in \eqref{anti4} and by using similar methods as in part $(i)$, we get
 \begin{equation}\label{anti5}
\delta(c^{\star})^{\star}    b + c    \delta(b)=\delta(c^{\star} b ^{\star})^{\star}+ b c    \eta,
 \end{equation}
 for all $ b , c  \in A $, where $\eta \in \A$ and $\delta(u_{i}) \xlongrightarrow{\sigma (\A, A^*)}\eta $. By \eqref{anti5} we have
 \begin{equation}\label{anti55}
  b^{\star}   \delta(c^{\star})+\delta(b)^{\star}    c^{\star}=\delta(c^{\star} b ^{\star})+\eta^{\star}    c^{\star} b^{\star},
  \end{equation}
 for all $ b , c  \in A$. Letting $c^{\star}=u_{i}$, we arrive at
 \[ b^{\star}    \eta+\delta(b)^{\star}=\delta(b^{\star})+\eta^{\star}    b^{\star},\]
 for all $b \in A$. Hence
 \begin{equation}\label{anti6}
\delta(b^{\star})-b^{\star}    \eta=(\delta(b)- b    \eta)^{\star},
 \end{equation} 
 for all $b  \in A $. From \eqref{anti55} we have
 \begin{equation}\label{anti555} 
 \delta(a b )= b    \delta(a)+\delta(b^{\star})^{\star}    a -\eta^{\star}    a b ,
 \end{equation}
 for all $ a , b  \in A$.
  Define $d: A \rightarrow \A$ by $d(a)=\delta(a)- a    \eta $. The linear map $d$ is continuous and from \eqref{anti6}, it follows that $d$ is a $\star$-map. Now from \eqref{anti555} we have
\begin{equation*}
\begin{split}
d(a^2) &=\delta(a^2)-a^2 \eta  \\
 &=a\delta(a)+\delta(a^{\star} )^{\star}a-\eta ^{\star} a^2-a^2 \eta\\
&= a(\delta(a)-a\eta)+(\delta (a^{\star} )^{\star}-(a^{\star}\eta)^{\star})a \\ 
&=  ad(a)+(d(a^{\star})^{\star})a  \\
&=  ad(a)+d(a)a
\end{split}
\end{equation*}   
for all $a\in A$. Thus $d$ is a continuous $\star$-derivation and by \cite{jon}, $ d $ is a $ \star $-derivation.
\end{proof}
\begin{cor} \label{cc}
Let $A$ be a $ C^{\star}  $-algebra, and let $\delta: A \rightarrow \A$ be a continuous linear map. Suppose that
\begin{equation*} 
 a^{\star} b=0 \Longrightarrow a^{\star}   \delta(b)+\delta(a)^{\star}   b=0\quad (a , b \in A).
 \end{equation*}
Then there is a continuous $\star$-derivation $ d: A \rightarrow \A $ and an element $ \eta \in \A $ such that $\delta(a)= d(a) + \eta    a  $ for all $ a \in A$. 
\end{cor}
\begin{proof}
Consider the continuous linear map $ \tau: A \rightarrow \A $ defined by $ \tau (a) = \delta(a^{\star})^{\star} $. 
It is easily seen that the map satisfies in conditions of Theorem \ref{anti}-$(ii)$. 
So there exists a continuous $\star$-derivation $ d : A \rightarrow \A $ and an element $ \eta_1 \in \A $ such that
 $ \tau (a) = d(a) + a \eta_1    $ for all $ a \in A$. 
Then $ \delta(a)= d(a) + \eta a  $ for all $ a \in A$, where $ \eta = \eta_1^{\star} $.
\end{proof}
Note that in part (ii) of Theorem \ref{anti} and above corollary, it is not necessary true that $ \eta \in \mathcal{Z} ( \A) $. For example suppose that $ \eta $ is not in $ \mathcal{Z} ( \A) $ and $ \eta^\star [ a , b ] + [a , b ] \eta = 0 $. Define $ \delta : A \rightarrow \A $ by $ \delta (a) = a   \eta  $. Then $ \delta $ satisfies in condition of Theorem \ref{anti}-$(ii)$, but $ \delta $ equals to sum of the zero derivation and $ a \eta  $, while $ \eta $ is not belong to $ \mathcal{Z} ( \A) $.
\begin{prop}
Let $ A $ be a von Neumann algebra or a simple $ C^{\star} $-algebra with unity. Suppose that $ \delta : A \rightarrow A $ is a continuous linear map. Then
\begin{enumerate}
\item[(i)] $ \delta $ satisfies 
\[ a b=0 \Longrightarrow b \delta(a)+\delta(b) a =0\quad (a , b  \in A).\]
if and only if there are elements $\mu,\nu \in A$ such that $\delta(a)= a \mu-\nu  a $, where $ \mu  - \nu \in \mathcal{Z}(A)$ and 
\[ [[a , b ],\mu]+2 [a , b ]  (\mu-\nu)=0, \]
for all $ a , b \in A$.
\item[(ii)] $ \delta $ satisfies 
\begin{equation*}
 a b ^{\star}=0 \Longrightarrow \delta(b)^{\star} a+b^{\star}  \delta(a)=0\quad (a , b \in A).
 \end{equation*}
if and only if there are elements $\mu, \nu\in A$ such that $\delta(a)=a  \nu-\mu a $ for all $a \in A$ and $Re\mu \in \mathcal{Z}(A)$ and 
\[ [[a, b ],\mu]+(\nu-\mu)^{\star}  [a , b ]+[a , b ]  (\nu-\mu)=0,\]
for all $a , b  \in A$ .
\item[(iii)] $ \delta $ satisfies  
\begin{equation*}
 a^{\star} b =0 \Longrightarrow \delta(b) a^{\star}+b  \delta(a)^{\star}=0\quad (a , b  \in A).
 \end{equation*}
if and only if there are elements $\mu, \nu\in A$ such that $\delta(a)=a  \mu-\nu a $ for all $a \in A$ and $Re\mu \in \mathcal{Z}(A)$ and 
\[ [[a , b ],\mu]+[a , b ]  (\mu-\nu)^{\star}+(\mu-\nu)  [ a, b ]=0,\]
for all $ a , b \in A $ .
\end{enumerate}
\end{prop}
\begin{proof}
$ (i) $
Let $ \delta $ satisfies $(i)$. By Theorem \ref{anti}-$(i)$, there is a continuous derivation $ d:A \rightarrow \A $ and an
element $ \eta \in \mathcal{Z} (\A) $ such that $ \delta (a) = d(a) + \eta   a $ for all $ a \in A $. Since $ d (1) =0 $, we have $ \eta \in \mathcal{Z} (A) $, and $ d $ is a derivation on $ A $. By the fact that every derivation
on $ A $ is inner, it follows that $ d (a) = a \mu - \mu a $ for all $ a \in A $, where $ \mu \in A $. Setting $ \nu = \mu - \eta $. So $ \delta (a) = a \mu - \nu a $ for all $ a \in A $ and $ \mu - \nu \in \mathcal{Z} (A) $.\\
Now by \eqref{anti3} and the fact that $d$ is a derivation we see that 
\begin{equation*}
\begin{split}
\delta(ab)+ a b \eta &= b  \delta(a)+\delta(b) a \\
 &=b   d(a)+ b  \eta  a+d(b)   a +\eta   b a \\
&= d( b a )+2 b a \eta \\ 
&= \delta(b a )+ b a \eta,
\end{split}
\end{equation*} 
for all $ a , b \in A $. So
\[  a b  \mu-\nu   a b + a b   \eta = b a  \mu-\nu   b a + b a   \eta,\]
and hence
\[  a b   \mu- \mu    a b +2 a b \eta= b a   \mu-\mu   b a +2 b a \eta,\] 
for all $ a , b \in A$. Therefore
\[ [[a , b ],\mu]+2 [a , b ]   (\mu-\nu)=0, \]
for all $a , b \in A$.
\par 
Conversely, suppose that there are elements $\mu,\nu \in A$ such that $\delta(a)= a \mu-\nu  a $, where $ \mu  - \nu \in \mathcal{Z}(A)$ and 
\[ [[a , b ],\mu]+2 [a , b ]  (\mu-\nu)=0, \]
for all $ a , b \in A$. From this equation for all $a , b \in A$ with $ab=0$, we have 
\[ba\mu -\mu ba +2ba(\mu - \nu )=0.\]
Since $ \mu  - \nu \in \mathcal{Z}(A)$, it follows that
\[0=ba\mu -\mu ba +2(\mu - \nu )ba=ba\mu +\mu ba -2\nu ba.\]
Therefore 
\begin{equation*}
\begin{split}
b\delta(a)+\delta(b)a &= ba\mu -\nu ba+b(\mu-\nu)a \\
 &= ba\mu -\nu ba+(\mu-\nu)ba\\
&= ba\mu +\mu ba -2\nu ba \\ 
&= 0,
\end{split}
\end{equation*} 
for all $a , b \in A$ with $ab=0$.
\par 
$ (ii) $ By Theorem \ref{anti}-(ii), there is a continuous $ \star $-derivation $ d:A \rightarrow \A $ and an
element $ \eta \in \A $ such that $ \delta (a) = d(a) + a\eta    $ for all $ a \in A $, and by using similar arguments as above, it follows that $ \eta \in A $ and $ d $ is a derivation on $ A $. The derivation $d$ is an inner derivation and by Remark \ref{der}, there is an element $\mu\in A$ with $Re\mu\in \mathcal{Z}(A)$, such that $d(a)= a \mu-\mu a $ for all $ a \in A $. \\
 Now by \eqref{anti555} and the fact that $d$ is a $\star$-derivation, we have 
  \[ \delta( a b )+\eta^{\star}   a b =\delta( b a )+\eta^{\star} b a ,\]
for all $ a , b  \in A$. So from the fact that $\delta (a) = d(a) + a\eta    =a \mu-\mu a +a\eta     $ for all $a\in A$, we have
\[  a b  \mu-\mu  a b +\eta^{\star} a b = b a  \mu-\mu b a +\eta^{\star} b a ,\]
and hence
\[ [[ a , b ],\mu]+\eta^{\star} [a , b ]+[ a , b ]  \eta =0,\]
for all $ a , b  \in A$. By setting $\nu=\mu+\eta $, we have $\delta(a)= a \nu-\mu   a $ and 
\[ [[a , b ],\mu]+(\nu-\mu)^{\star} [ a , b ]+[ a , b ]  (\nu-\mu)=0,\]
for all $a , b  \in A$, where $Re\mu\in \mathcal{Z}(A)$.
\par 
Conversely, suppose that there are elements $\mu, \nu\in A$ such that $\delta(a)=a  \nu-\mu a $ for all $a \in A$ and $Re\mu \in \mathcal{Z}(A)$ and 
\[ [[a, b ],\mu]+(\nu-\mu)^{\star}  [a , b ]+[a , b ]  (\nu-\mu)=0,\]
for all $a , b  \in A$ . By this equation for all $a , b \in A$ with $ab^\star=0$, we have 
\[- \mu b^\star a+ \nu ^{\star} b^\star a -\mu^{\star} b^{\star}a +b^{\star} a\nu=0.\]
So by $Re\mu \in \mathcal{Z}(A)$ we get
\begin{equation*}
\begin{split}
b^{\star}\delta(a)+\delta(b)^{\star}a &= b^{\star} a\nu + \nu^{\star} b^{\star}a-b^{\star}(\mu+\mu^{\star})a \\
 &= b^{\star} a\nu + \nu^{\star} b^{\star}a-\mu b^{\star}a-\mu^{\star}b^{\star}a \\
&= 0,
\end{split}
\end{equation*} 
for all $a , b \in A$ with $ab^{\star}=0$.
\par 
$ (iii) $ Define the map $d: A \rightarrow A$ by $d(a)=\delta(a^{\star})^{\star}$. It is easily seen that the map $d$ satisfies in conditions of part $(ii)$. So, there exists $\mu_{1}, \nu_{1}\in A$ such that $d(a)=a \nu_{1}-\mu_{1} a$ for all $a \in A$ with $Re\mu_{1}\in \mathcal{Z}(A)$ and 
\[ [[a , b ],\mu_{1}]+(\nu_{1}-\mu_{1})^{\star} [a , b ]+[a , b ] (\nu_{1}-\mu_{1})=0,\]
for all $a ,b  \in A $. Then $\delta(a)=a   \mu-\nu   a $ for all $a \in A $, where $\nu=-\nu_{1}^{\star}$, $\mu=-\mu_{1}^{\star}$ with $Re\mu\in \mathcal{Z}(A)$ and
\[ [[a , b ],\mu]+[a , b ] (\mu-\nu)^{\star}+(\mu-\nu)[a , b ]=0,\]
for all $a , b  \in A$.
\par 
The converse is proved by a similar method as in part $(ii)$.
\end{proof}
\section{Derivations through two-sided orthogonality conditions}
In this section we will consider a linear map behaving like derivation at two-sided orthogonality conditions.
In order to prove our results we need the following result.
\begin{lem} \label{l32}
Let $A$ be a unital $ C^{\star}  $-algebra, let $ X $ be a Banach space, and let $ \phi: A \times A \rightarrow X $ be a continuous bilinear map satisfying
\[ ab = ba = 0 \Longrightarrow \phi ( a, b ) =0 ~~~~~~~~ (a, b \in A), \]
Then for all $ a , c \in \mathcal{Z} (A) $ and $ b \in A $ we have
\[ \phi (ab , c ) = \phi (a , b c ). \]
\end{lem}
\begin{proof}
Let $ a, c \in \mathcal{Z} (A)  $, and $b$ be a self-adjoint element i $A$. Let $I$ be a compact interval of $ \mathbb{R} $ containing the spectrum of $ b $. Define $ \psi : C(I) \times C(I) \rightarrow X $ by
\[ \psi(f, g) = \phi (af(b) , g (b) c ). \]
If $f, g \in  C(I) $ are such that $fg = 0$, then $ a f(b) g(b) c = g(b) c a f (b) = g(b) f (b)ca=0 $ and so $ \psi(f, g) = 0$. On account of Lemma \ref{l21}, $\psi$ is symmetric, i.e. $ \psi (f, g) = \psi(g, f) $ for all $f, g \in C(I)$, hence
\[ \phi ( a f (b) , g(b) c ) = \phi ( a g (b) , f(b) c ) . \]
Now, for $f(s) = 1$ and $g(s) = s$ we obtain $ \phi ( a , bc ) = \phi ( a b, c ) $, which readily implies the desired conclusion (since $b$ is self-adjoint).
\end{proof}
\par
In continue we give the main results of this section.
\par 
\begin{prop}\label{p33}
Let $A$ be a $C^\star$-algebra, and let $\delta: A \rightarrow \A $ be a continuous linear map. Assume that 
\[  a b = b a =0 \Longrightarrow a   \delta(b)+\delta(a)    b =b    \delta(a)+\delta(b)    a =0\quad (a , b  \in A).\]
Then there is a continuous derivation $ d: A \rightarrow \A $ and an element $ \eta \in \mathcal{Z}(\A) $ such that $\delta(a)= d(a) + \eta    a  $ for all $ a \in A$.
\end{prop}
\begin{proof}
For all $ a , b \in A $ with $ ab=ba = 0  $, we have 
\[a   \delta(b)+\delta(a)    b + b    \delta(a)+\delta(b)    a =0 . \]
So by \cite [Theorem 4.1]{al10}, there exist a continuous derivation $ d : A \rightarrow \A $ and a bimodule homomorphism $ \Phi : A \rightarrow \A $ such that $ \delta = d + \Phi $.\par 
Suppose that $(u_i)_{i\in I}$ is a bounded approximate identity of $A$, By \cite{jon0}, $ \Phi $ is continuous and so the net $\{ \Phi(u_{i}) \}_{i\in I}$ in $ \A $ is bounded. Hence we can assume that $\Phi(u_{i}) \xlongrightarrow{\sigma (\A, A^*)} \eta $ for some $\eta \in \A$. Now, for all $ a \in A $ we have $ u_i a \rightarrow a $ and $ a u_i \rightarrow a  $. Thus $ \Phi (u_i a) \rightarrow \Phi (a) $ and $ \Phi (a u_i ) \rightarrow \Phi (a) $. On the other hand by separately $w^{*}$-continuity of product in $\A$, we see that 
\[ \Phi (u_i) a \xlongrightarrow{\sigma (\A, A^*)} \eta a \quad \text{and} \quad a \Phi (u_i) \xlongrightarrow{\sigma (\A, A^*)} a \eta ,\]
for all $a\in A$. Since $\Phi$ is a bimodule homomorphism, it follows that 
\[ \Phi (a) = \eta    a = a   \eta,  ~~~~ (a \in A, ) ,\] 
and so
\[ \delta (a) = d (a) + \eta   a , \] for all $a\in A$. Also from Remark \ref{12}, $ \eta \in \mathcal{Z} (\A ) $. It is clear that $d$ is continuous.
\end{proof}

\begin{thm} \label{t43}
Let $A$ be a unital $C^\star$-algebra, and let $\delta: A \rightarrow \A $ be a continuous linear map. Assume that 
\begin{equation*}
 ab^{\star}=b^{\star} a=0 \Longrightarrow a    \delta(b)^{\star}+\delta(a)   b^{\star}=\delta(b)^{\star}    a + b ^{\star}    \delta(a)=0\quad (a , b  \in A).
 \end{equation*}
Then there are continuous $ \star $-derivations $ d_1 , d_2 : A \rightarrow \A $ and an element $ \eta \in \A $ with $ Re \eta \in \mathcal{Z} ( A) $ such that $\delta(a)= d_1(a) + \eta    a = d_2 (a) + a   \eta $ for all $ a \in A$.
\end{thm}
\begin{proof}
Define continuous bilinear maps $\phi, \psi:A \times A \rightarrow \A$ by 
\[\phi(a , b )= a    \delta(b^\star)^\star+\delta(a)    b  \, \,\, \text{and} \,\,\, \psi(a , b )= \delta(b^\star)^\star   a + b   \delta(a) \]
 It is easily seen that $ \phi (a , b)= 0  $ and $ \psi (a , b ) = 0 $, whenever $ a,b \in A $ are such that $ ab= ba=0 $. By Lemma \ref{l32} we have
\begin{equation} \label{I}
ab   \delta (c^\star)^\star + \delta (ab)   c = a   \delta (c^\star b^\star)^\star + \delta (a)   bc ,
\end{equation}
\begin{equation} \label{II}
\delta (c^\star)^\star    ab + c   \delta (ab)  =  \delta (c^\star b^\star)^\star    a + bc    \delta (a) 
\end{equation}
for all $ a , c \in \mathcal{Z} (A) $ and $ b \in A $. Now letting $a = 1$ and $ \eta = \delta (1) $ in \eqref{I} and \eqref{II}, we obtain
\[ b   \delta (c^\star)^\star + \delta (b)   c = \delta (c^\star b^\star)^\star + \eta   bc , \]
\[ \delta (c^\star)^\star   b + c   \delta (b)  =  \delta (c^\star b^\star)^\star  + bc   \eta \]
for all $ c \in \mathcal{Z} (A) $ and $ b \in A $. By applying the $\star$ on above equations, and setting $ c^\star = 1 $, we arrive at
\begin{equation} \label{III}
\eta    b^\star + \delta (b)^\star = \delta (b^\star) + b^\star   \eta ^\star ,
\end{equation}
and
\begin{equation} \label{IV}
b^\star   \eta + \delta (b)^\star  = \delta ( b^\star)  + \eta^\star   b^\star,
\end{equation}
for all $ b\in A $. By \eqref{III} and \eqref{IV}, we get $ ( \eta + \eta^\star )   b^\star = b^\star (\eta + \eta^\star) $ for all $ b\in A $. Therefore by Remark \ref{12}, $ Re \eta \in \mathcal{Z} (\A) $. Define the map $ d_1:A \rightarrow \A $ by $ d_1(a) = \delta (a) - \eta   a $. Then $d_1$ is a continuous linear map, and by \eqref{III} we have $ d_1 (a^\star) = d_1 (a)^\star $ for all $ a \in A $. So $d$ is a $ \star $-map. If $ ab = ba = 0  $, then by hypothesis, definition of $d_1$ and the fact that $d_1$ is a $ \star $-map and $ Re \eta \in \mathcal{Z} (\A) $, we have 
\begin{equation*}
\begin{split}
a    d_1(b)+d_1(a )    b &= a    d_1(b^{\star})^{\star}+d_1(a)     b \\ &=
a     (\delta(b^{\star})-\eta    b^{\star})^{\star}+(\delta(a)-\eta    a)    b=0,
\end{split}
\end{equation*}
and 
\begin{equation*}
\begin{split}
b    d_1(a)+ d_1(b)    a &= b    d_1(a)+d_1 (b^{\star})^{\star}     a \\ 
&= b     (\delta(a)-\eta    a)+(\delta(b^{\star})-\eta    b^{\star})^{\star}    a\\
&= -b     \eta    a - b    \eta^{\star}    a=-b a    (\eta+\eta^{\star})=0.
\end{split}
\end{equation*}
So $d_1$ satisfies in conditions of Proposition \ref{p33} and hence there exist a continuous derivation $ d: A \rightarrow \A $ and an element $ \eta \in \mathcal{Z} (\A) $ such that 
\[ d_1 (a) = d (a) + \eta  a ~~~~ (a \in A) . \]
By definition of $d_1$, we have $ d_1 (1) = 0  $, on the other hand $d$ is a derivation, so $ d (1) = 0  $. Hence $ \eta = 0 $ and $ d = d_1 $. Hence $d_1$ is a $ \star $-derivation and so
\[ \delta (a) = d_1 (a) + \eta   a ~~~~ (a \in A) . \]
with $ Re \eta \in \mathcal{Z} (\A) $. By defining  $ d_2 (a) = \delta (a) - a   \eta $ and by using similar arguments as above, it follows that $ d_2 $ is a continuous $ \star $-derivation.
\end{proof}
If $ \delta(1) \in \mathcal{Z} (\A)  $, it is obvious that $d_1 = d_2$ in above theorem. Indeed, in this case there is a  derivation $ d: A \rightarrow \A $ such that $ \delta(a) = d (a) + \delta (1)   a $ for all $ a \in A $. Note that in this theorem, it is not necessary true that $ \eta \in \mathcal{Z} ( \A) $. For example suppose that $ \eta \in \A $, $ \eta $ is not in $ \mathcal{Z} ( \A) $ and $ Re\eta \in \mathcal{Z} (\A ) $. Define $ \delta : A \rightarrow \A $ by $ \delta (a) = \eta   a  $. Then $ \delta $ satisfies in conditions of above theorem, but $ \delta $ equals to sum of the zero derivation and $ \eta   a  $, while $ \eta $ is not belong to $ \mathcal{Z} ( \A) $.
\par
In the next proposition we consider von Neumann algebras or simple $ C^{\star} $-algebras with unity.
\begin{prop} \label{c53}
Let $ A $ be a von Neumann algebra or a simple $ C^{\star} $-algebra with unity. Suppose that $ \delta : A \rightarrow A $ is a continuous linear map. Then 
\begin{enumerate}
\item[(i)] $ \delta $ satisfies 
\[  a b = b a =0 \Longrightarrow a   \delta(b)+\delta(a)    b =b    \delta(a)+\delta(b)    a =0\quad (a , b  \in A).\]
if and only if there are elements $\mu,\nu \in \A$ such that $\delta(a)=a     \mu-\nu    a $ for all $a \in A$, where $\mu-\nu\in \mathcal{Z}(\A)$.
\item[(ii)] $ \delta $ satisfies 
\[ ab^{\star}=b^{\star} a=0 \Longrightarrow a    \delta(b)^{\star}+\delta(a)   b^{\star}=\delta(b)^{\star}    a + b ^{\star}    \delta(a)=0\quad (a , b  \in A).
 \]
if and only if there are elements $\mu, \nu\in \A$ such that $\delta(a)=a     \mu-\nu    a $ for all $ a  \in  A $, where $Re\mu\in \mathcal{Z}(\A)$ and $Re(\mu-\nu)\in \mathcal{Z}(\A)$.
\end{enumerate}
\end{prop}
\begin{proof}
$ (i) $ By Proposition \ref{p33}, there exist a continuous derivation $ d : A \rightarrow \A $ and an element $ \eta \in \mathcal{Z} (\A) $ such that 
\[ \delta (a) = d (a) + \eta   a ~~~~ (a \in A) . \]
By using similar arguments as above, it follows that $ \eta \in \mathcal{Z} (A) $ and $ d : A \rightarrow A $ is a derivation. So $ d $ is inner, and $ d (a) = a \mu - \mu a $ for all $ a \in A $ whenever $ \mu \in A $. Setting $ \nu = \mu - \eta $. Thus $ \delta (a) = a \mu - \nu a $ for all $ a \in A $ and $ \mu - \nu \in \mathcal{Z} (A) $.

The converse is proved easily.
\par 
$ (ii) $ By Theorem \ref{t43}, there exist a continuous $ \star$-derivation $ d_1 : A \rightarrow \A $ and an element $ \eta \in \mathcal{Z} (\A) $ such that $ \delta (a) = d_1(a) + \eta   a $ for all $ a \in A $.
By using similar arguments as above, it follows that $ \eta \in A $ whenever $ Re\eta \in \mathcal{Z} (A) $ and $ d_1 : A \rightarrow A $ is a $ \star$-derivation. Therefore $ d_1 $ is inner, and there is an element $ \mu \in A $ with $ Re \mu \in \mathcal{Z} (A) $ such that  $ d_1 (a) = a \mu - \mu a $ for all $ a \in A $. Taking $ \nu = \mu - \eta $. Thus $ \delta (a) = a \mu - \nu a $ for all $ a \in A $ whenever $ Re \mu \in \mathcal{Z} (A) $ and $Re( \mu - \nu) \in \mathcal{Z} (A) $.
\par 
Conversely, suppose that there are elements $\mu, \nu\in \A$ such that $\delta(a)=a     \mu-\nu    a $ for all $ a  \in  A $, where $Re\mu\in \mathcal{Z}(\A)$ and $Re(\mu-\nu)\in \mathcal{Z}(\A)$.
\par 
By $Re\mu\in \mathcal{Z}(\A)$, for $a,b\in A$ with $ab^{\star}=b^{\star}a=0$ we have
\begin{equation*}
\begin{split}
a\delta(b)^{\star}+\delta(a)b^{\star} &= a(\mu+\mu^{\star})b^{\star} \\
 &= (\mu+\mu^{\star})ab^{\star} \\ 
&= 0.
\end{split}
\end{equation*} 
Also for $a,b\in A$ with $ab^{\star}=b^{\star}a=0$, by $Re\mu\in \mathcal{Z}(\A)$ and $Re(\mu-\nu)\in \mathcal{Z}(\A)$ we arrive at 
\begin{equation*}
\begin{split}
\delta(b)^{\star}a+b^{\star}\delta(a)&= -b^{\star} (\nu + \nu^{\star})a\\
 &=b^{\star}(\mu+\mu^{\star})a -b^{\star} (\nu + \nu^{\star})a \\ 
 &=b^{\star}(\mu-\nu)a +b^{\star}(\mu-\nu)^{\star}a\\
 &=(\mu-\nu)b^{\star}a +(\mu-\nu)^{\star}b^{\star}a\\
&= 0.
\end{split}
\end{equation*} 
\end{proof}
Note that in Proposition \ref{c53} the converses of Results \ref{p33} and \ref{t43} hold, but we do not know the converses of these results are true or not, in general.

\subsection*{Acknowledgment}
The author thanks the referee for careful reading of the manuscript and for helpful suggestions.

\bibliographystyle{amsplain}
\bibliography{xbib}

\begin{thebibliography}{20}



\bibitem{al} J. Alaminos, M. Bre$\check{\textrm{s}}$ar, J. Extremera and A. R. Villena, \textit{Characterizing homomorphisms and derivations on $C^{\star}$-algebras}, Proc. Roy. Soc. Edinburgh Sect. A 137 (2007), 1--7.

\bibitem{al1} J. Alaminos, M. Bre$\check{\textrm{s}}$ar, J. Extremera and A. R. Villena, \textit{Maps preserving zero products}, Studia Math. 193 (2009), 131--159.

\bibitem{al10} J. Alaminos, M. Bre$\check{\textrm{s}}$ar, J. Extremera and A. R. Villena, \textit{Characterizing Jordan maps on $C^{\star}$-algebras through zero products}, Proc. Edinburgh Math. Soc. 53 (2010), 543--555.




\bibitem{bre} M. Bre$\check{\textrm{s}}$ar, \textit{Characterizing homomorphisms, derivations and multipliers in rings with idempotents}, Proc. R. Soc. Edinb. Sect. A. 137 (2007), 9--21.

\bibitem{chb} M.A. Chebotar, W.-F. Ke and P.-H. Lee, \textit{Maps characterized by action on zero products}, Pacific. J. Math. 216 (2004), 217--228.

\bibitem{che} H--Y. Chen, K--S. Liu and M. R. Mozumder, \textit{Maps acting on some zero products}, Taiwanese J. Math. 18 (2014), 257--264.

\bibitem{da} H.G. Dales, \textit{Banach algebras and automatic continuity}.  London Math. Soc. Monographs. Oxford Univ. Press, Oxford 2000.

\bibitem{ess}A.B.A. Essaleh and A.M. Peralta, \textit{Linear maps on $C^{\star}$-algebras which are derivations or triple derivations at a point}, Linear Algebra Appl. 538 (2018), 1--21. 




\bibitem{gh1} H. Ghahramani, \textit{Additive mappings derivable at nontrivial idempotents on Banach algebras}, Linear and Multilinear Algebra, 60 (2012), 725--742.

\bibitem{gh12} H. Ghahramani, \textit{Additive maps on some operator algebras behaving like $(\alpha,\beta)$-derivations or generalized $(\alpha,\beta)$-derivations at zero-product elements}, Acta Mathematica Scientia, 34B(4) (2014), 1287--1300.

\bibitem{gh2} H. Ghahramani, \textit{On Centralizers of Banach Algebras}, Bull. Malays. Math. Sci. Soc. 38(1) (2015), 155--164.

\bibitem{h5} B. Fadaee and H. Ghahramani, \textit{Jordan left derivations at the idempotent elements on reflexive algebras}, Publicationes mathematicae-Debrecen, 92/3-4 (2018), 261--275. 

\bibitem{h6} H. Ghahramani and Z. Patrick Pan, \textit{Linear maps on *-algebras acting on arthogonal elements like derivations or anti-derivations}, Filomat, 32:13 (2018), 4543--4554.

\bibitem{gh5} 
H. Ghahramani, \textit{Linear maps on group algebras determined by the action of the derivations or anti-derivations on a set of orthogonal elements}, Results in Mathematics, 73 (2018), 132--146. 

\bibitem{jin} W. Jing , S. Lu S and P. Li. \textit{Characterization of derivations on some operator algebras}, Bull Austr Math Soc, 66 (2002), 227--232.

\bibitem{jon0} B.E. Johnson, \textit{Continuity of centralizers on Banach algebras}, J. London Math. Soc. 41 (1996), 639--640.

\bibitem{jon} B.E. Johnson, \textit{Symmetric amenability and nonexistence of Lie and Jordan derivations}, Math. Proc. Camb. Phil. Soc. 120 (1996), 455--473.

\bibitem{sak} S. Sakai, \textit{$C^{\star}$-algebras and $W^{\star}$-algebras}, Springer Verlag, Berlin, Heidelberg and New York.

\bibitem{zhang1}Y.F. Zhang, J.C. Hou and X.F. Qi, \textit{Characterizing derivations for any nest
algebras on Banach space by their behaviors at an injective operator}, Linear Algebra
Appl., 449 (15) (2014), 312--333. 

\bibitem{zh1} J. Zhu, \textit{All-derivable points of operator algebras}, Linear Algebra Appl. 427 (2007), 1--5. 



\bibitem{zh3}J. Zhu and Ch. Xiong, P. Li, \textit{Characterizations of all-derivable points in B(H)},
Linear Multilinear Algebra 64, no. 8 (2016), 1461--1473.

\bibitem{zh4} J. Zhu and S. Zhao, \textit{Characterizations all-derivable points in nest algebras}, Proc.
Amer. Math. Soc. 141 (7) (2013), 2343--2350. 









\end{thebibliography}

\end{document}